\documentclass[a4paper,12pt]{article}

\usepackage{epsfig}
\usepackage{amsfonts,eucal}
\usepackage{amssymb,amsmath,amsthm,color}
\usepackage{authblk}
\usepackage[T2A]{fontenc}
\usepackage[cp1251]{inputenc}

\newtheorem{theorem}{Theorem}[section]
\newtheorem{proposition}{Proposition}[section]
\newtheorem{lemma}{Lemma}[section]

\newcommand{\R}{{\mathbb R}}

\newcommand{\X}{{\R^d}}

\newcommand{\eps}{\varepsilon}

\begin{document}

\title{Homogenization of parabolic problems for non-local convolution type operators under non-diffusive scaling of coefficients.}

\setcounter{footnote}3

\author[1,2,$\ast$]{A. Piatnitski
}
\author[1,2,$\star$]{E. Zhizhina}
\affil[1]{\small The Arctic University of Norway, Campus Narvik,

P.O.Box 385, 8505 Narvik, Norway}
\affil[2]{\small Higher School of Modern Mathematics, MIPT,

1 Klimentovskiy per., 115184 Moscow, Russia}
\affil[$\ast$]{ email: {\tt apiatnitski@gmail.com} }
\affil[$\star$]{email: {\tt elena.jijina@gmail.com}}
\date{}

\maketitle

{\parindent 6.6cm \it
 Dedicated to  Lars-Erik Persson,

a famous mathematician

 and an active skier}

\bigskip
\begin{abstract}
We study homogenization problem for non-autonomous  parabolic equations of the form $\partial_t u=L(t)u$ with an integral convolution type operator $L(t)$ that has
 a non-symmetric jump kernel which is periodic in spatial variables and in time.
 It is assumed that the space-time scaling of the environment is not diffusive.
We show that asymptotically the spatial and temporal evolutions of the solutions are getting decoupled,
and the homogenization result holds in a moving frame.

\end{abstract}


\section{Introduction and previous results}\label{Intro}

This paper deals with  homogenization problem for a parabolic type equation of the form
\begin{equation}\label{eq_eps}
\frac{\partial u^\varepsilon}{\partial t} \ = \ \frac{1}{\varepsilon^{d+2}} \int\limits_{\X} a\Big(\frac{x-y}{\eps}\Big) \ \mu \Big(\frac{x}{\eps}, \frac{y}{\eps}; \frac{t}{\varepsilon^\alpha} \Big)\ (u^\eps (y,t) - u^\eps (x,t)) dy,
\end{equation}
here $\varepsilon>0$ is a small parameter, and $0< \alpha<2$.  We assume that  $a(z)$ is a non-negative,
integrable function that has a finite second moment, while  the coefficient $ \mu (\xi, \eta; s)$ is periodic both in spatial variables $\xi$ and $\eta$ and in time $s$. It is also assumed that $\mu(\cdot)$ is  bounded and positive definite.
It should be emphasized that  we do not impose symmetry conditions in $\xi$ and $\eta$ on  $\mu(\xi,\eta,s)$ and $a(\xi-\eta)$. 

Our goal is to study 
the  asymptotic behaviour of solutions to the Cauchy problem for equation \eqref{eq_eps} with the initial
condition
\begin{equation}\label{ini_eps}
  u^\eps(x,0)=u_0(x),\quad u_0\in L^2(\mathbb R^d),
\end{equation}
as $\eps\to0$.
It turns out that asymptotically, as $\eps\to 0$,  the spatial evolution of solutions to the said Cauchy problem
and their temporal evolution are getting decoupled. More precisely, we show that there exists a function $b^\eps(t)$ with values in $\mathbb R^d$ such that in the moving frame
$$
(x,t)\,\to\, \Big(x - b^\eps (t),\,t\Big)=:({x}^\eps,t),
$$
the solution  $u^\eps({x}^\eps,t)$  converges as $\eps \to 0$  to a solution of the Cauchy problem
for a heat equation:
\begin{equation}\label{intr_eff}
\frac{\partial u}{\partial t} \ = \ \mathrm{div}\big(\Theta\nabla u\big),\quad  u(x,0)=u_0(x),
\end{equation}
with a constant positive definite matrix $\Theta$.

The function $b^\eps(t)$ is of particular interest. If the coefficient $\mu(\xi,\eta,s)$ is sufficiently regular in $s$,
then
\begin{equation}\label{b_eps_def}
b^\eps(t)=\sum\limits_{j=0}^{k(\alpha)}
\eps^{-1+j(2-\alpha)} b_{j}t+
\eps^{\alpha-1} B_0(\frac t{\eps^\alpha}), 
\end{equation}
where $k(\alpha)=\left[\frac1{2-\alpha}\right]$. Here $b_0,\,b_1,\ldots$ are vectors in $\mathbb R^d$,
$B_0(s)$ is a $1$-periodic continuously differentiable vector-function, and $[\cdot]$ stands for the integer part.

In the existing literature there is a number of homogenization results obtained in moving coordinates.
We quote here the works \cite{DP} and \cite{G}, where non-autonomous parabolic problems for second order elliptic differential operators
with large low order terms  were investigated in periodic media.
In \cite{DP} it was proved that, under the diffusive scaling,  the homogenization result holds in rapidly moving coordinates and, in the presence of zero order term, homogenization takes place after factorization of solutions with the ground state of the corresponding periodic  cell problem.

The paper \cite{G} focuses on  homogenization problems for convection-diffusion operators of the form
$L_\eps u(x,t)=\Delta u(x,t)+b\big(\frac t{\varepsilon^p},\frac x\varepsilon\big)\nabla u(x,t) $ with  $b(\xi,s)$
being  periodic both in spatial variables and time, and $0<p<2$, that is the scaling is not
diffusive.  The authors considers the case of a potential vector field $b(\xi,s)$:  $b(\xi,s)=\nabla_\xi U(\xi,s)$ with
periodic in $\xi$ and $s$ function  $U$ and proved by the probabilistic methods that the homogenization result holds in moving coordinates.

In the work \cite{ZhiPa16}, Sections 12, 13,  the authors consider convection-diffusion operators in periodic media.
They show that for the corresponding semigroups the homogenization result holds in moving coordinates and
obtain estimates for the rate of convergence in operator norms.

The case of non-stationary convection-diffusion equations with a periodic microstructure whose characteristics are
rapidly oscillating random stationary functions of time is addressed in \cite{KP}.

\medskip


\medskip
Recent years there is a growing interest in nonlocal convolution type equations in inhomogeneous environments.
The equations of this type appear in various problems of population biology,  material sciences and porous media.
When modeling nonlocal processes in environments with a microstructure, we face the problem of a macroscopic
description of such processes. This leads to homogenization problems for the corresponding equations.

Homogenization problems for parabolic convolution-type equations in periodic environments are considered in
 \cite{AA} and \cite{PZh2023}.

The paper  \cite{AA} deals with the Cauchy problems for an equation of the form \eqref{eq_eps} with the initial condition
$u_0\in L^2(\mathbb R^d)$ under the assumption that
$\mu(x,y)$ is periodic in $x$ and $y$, does not depend on time and is not necessary symmetric.
Under natural coerciveness and moment conditions, it is shown that the said Cauchy problem admits homogenization in moving coordinates:
$$
u^\eps(x,t)=u^0\Big(x-\frac b\eps t,t\Big)+ R^\eps(x,t), \quad\hbox{where }\
\lim\limits_{\eps\to0}\|R^\eps\|_{L^\infty(0,T;L^2(\mathbb R^d))}=0,
$$
and $u^0(x,t)$ is a solution of \eqref{intr_eff}
with the same initial condition $u_0$.

Notice that in the cited articles the evolution of the moving frame is linear in time.

\medskip
If the coefficients $\mu(\xi,\eta,s)$ and $a(\xi-\eta)$ in \eqref{eq_eps} are symmetric in $\xi$ and $\eta$, and the scaling is diffusive that is  $\alpha=2$,
then, as was shown in the recent work \cite{PZh2023}, a solution of \eqref{eq_eps}--\eqref{ini_eps} converges to
a solution  of homogenized problem \eqref{intr_eff} with a positive definite matrix $\Theta$. In this case
the usual homogenization result holds.

It should be emphasized that in this case 
the equation for the first corrector is parabolic, it reads
$$
\partial_s\chi(\xi,s)-\!\int\limits_{\mathbb T^d}a(\xi-\eta)\mu(\xi,\eta,s)\big(\chi(\xi,s)-\chi(\eta,s)\big)d\eta=\!
\int\limits_{\mathbb T^d}a(\xi-\eta)\mu(\xi,\eta,s)(\xi-\eta)d\eta.
$$
According to  \cite{PZh2023} this equation has a unique up to an additive constant periodic in $s$ solution. The function $\chi$
is then used for constructing the asymptotic expansion of a solution of the original equation and, after substitution in the equation,   allows to eliminate all the growing terms. Thus, we need only one first order corrector and one second order corrector in the asymptotic expansion.

Homogenization with non-diffusive scaling and nonlocal effects extends classical theory to systems where microscopic
interactions are long-ranged and scaling laws deviate from diffusive regimes.

In the present paper we consider nonlocal convolution type equations of the form \eqref{eq_eps} in the case of non-diffusive scaling when the oscillation in time is somehow slower than that in spatial variables. Also we do not assume the symmetry of the coefficient.
In this case, the homogenization  procedure has two interesting features.

Firstly, the homogenization result holds in a moving frame whose evolution is not linear in time,
it is the sum of  linear and  periodic functions.

Secondly, due to non-diffusive scaling in \eqref{eq_eps}, the auxiliary problem for the corrector is not parabolic
any more.  Instead, we freeze the value of time and solve the corresponding auxiliary elliptic equation for each value of time,
see equation \eqref{Fcorr} below. This gives us the first corrector.  However, this corrector is not sufficient for eliminating all the growing terms of the discrepancy. In order to fix this problem, we construct a finite series of first order correctors, the number of terms in this series depends on $\alpha$, it is growing as $\alpha$ approaches $2$. Unfortunately, this construction requires
some regularity of the coefficient $\mu(\xi,\eta,s)$ with respect to the variable $s$. The asymptotic  behaviour of solutions
for non-regular $\mu$ is an open problem.



The paper is organized as follows. In Section \ref{sec_setup} we provide the conditions on the coefficients
of the operator $A(s)$ and formulate our main results.

Section \ref{Proof-1} deals with auxiliary cell problems and correctors. First we introduce a formal asymptotic
expansion of the solution, substitute this expansion in the equation and collect power-like terms in the resulting
relation. This allows us to define auxiliary problems and to construct correctors and a moving frame.

In Section \ref{sec_proofttt} we obtain a priori estimates for the solutions of the original problems and prove the main results.


\section{Problem setup and main results}\label{sec_setup}

The present work deals with  homogenization of Cauchy problem  for non-autonomous  parabolic convolution-type
equations 
in periodic in space and time media under a non-diffusive scaling of the coefficients. This equation reads
\begin{equation}\label{ANA_eps}
\begin{array}{l}
\displaystyle
H^\varepsilon u (x,t) := \partial_t u (x,t) -  L^{\varepsilon}(t) u (x,t) = 0, 
\quad \mbox{where} \\[3mm]
\displaystyle
(L^\varepsilon(t) u)(x,t) \ = \ \frac{1}{\varepsilon^{d+2}} \int\limits_{\mathbb R^d} a\big(\frac{x-y}{\eps}\big) \mu\big(\frac{x}{\eps}, \frac{y}{\eps}; \frac{t}{\eps^\alpha}\big) (u(y,t) - u(x,t)) dy.
\end{array}
\end{equation}
Here $0<\alpha<2$, and $a(z)$ satisfies the following conditions:
\begin{equation}\label{M1}
  a(z)\geqslant 0, \quad \int_{\mathbb R^d}a(z)dz=1,\quad \int_{\mathbb R^d}|z|^2a(z)dz<\infty.
\end{equation}
We assume that $ \mu(x,y; t)$ is periodic in each of the variable $x$, $y$ with period $[0,1)^{d}$ and in $t$ with period $1$, and that
 there exist positive constants $\mu_-$ and $\mu_+$ such that
\begin{equation}\label{lm-random}
0<\mu_- \le  \mu (x,y; t) \le \mu_+ \quad \mbox{ for all } \; x,y, \in \mathbb{R}^d, \; t \in \mathbb{R}.
\end{equation}
Moreover, we assume that  $\mu (x,y; t)$ is a $k+1$ times differentiable in the variable $t$ function:
  \begin{equation}\label{k_reg_mu}
\mu (x,y; t) \in C^{k+1}(\mathbb T^1; L^\infty(\mathbb T^{2d}))
\end{equation}
with $k=k(\alpha)= \big[ \frac{1}{2-\alpha} \big]$.

Denote by $u^{\varepsilon}(x,t)$ a solution of the Cauchy problem
\begin{equation}\label{th-2}
\partial_t u^\eps (x,t) = L^{\varepsilon}(t) u^\eps (x,t), \quad u^\eps(x,0) = u_0(x), \quad u_0 \in L^2(\mathbb R^d),
\end{equation}
$t \in [0, T]$, where the operator $ L^{\varepsilon}$ is defined in  \eqref{ANA_eps}.
Since by the Schur lemma   the norm of the operator $ L^{\varepsilon}(t)$ satisfies the estimate
$ \|L^{\varepsilon}(t)\|_{L^2(\mathbb R^d)\to L^2(\mathbb R^d)}\leqslant 2\eps^{-2}\mu_+\|a\|_{L^1(\mathbb R^d)}$,
see \cite[Theorem 5.2]{Halmos}, then for each $\eps>0$ problem  \eqref{th-2} has a unique solution $u^\eps\in L^\infty(0,T; L^2(\mathbb R^d))$.
The main result of the present work is the following statement:

\begin{theorem}\label{MT}
Let $a(z)$ satisfy all the conditions in \eqref{M1},
and assume that $\mu (x,y,t)$ is periodic in each of the variables  $x$, $y$ with period $[0,1)^{d}$ and in $t$
with period $1$. Assume moreover that  $\mu (x,y,t)$
satisfies conditions \eqref{lm-random}, \eqref{k_reg_mu}.
Then there exist a positive definite symmetric matrix $\Theta$ and a family of functions
$b^\eps (t) \in C([0,+\infty); \mathbb R^d)$ (see
Sections \ref{sec3_1} and \ref{sss41} for their definitions)  such that for any $T>0$
\begin{equation}\label{th-4-I}
\| u^{\varepsilon} ( x + b^{\eps}(t) ,\,t ) -
u^0 (x, t)\|_{L^{\infty}((0,T),\ L^2(\mathbb R^d) )}  \to 0, \quad \mbox{as } \; \varepsilon \to 0,
\end{equation}
where $u^0(x,t)$ is a solution of the Cauchy problem
\begin{equation}\label{th-3}
\frac{\partial u}{\partial t} = \mathrm{div}\big(\Theta \nabla u\big), \quad u(x,0) = u_0(x), \quad u_0 \in L^2(\mathbb R^d), \quad t \in [0, T].
\end{equation}
The profile of the function $b^\eps(t)$ depends on whether $0<\alpha<1$, or $\alpha=1$, or $1<\alpha<2$.
If $0<\alpha<1$, then 
$$
b^\eps(t)=\frac {b_0}\eps t+\eps^{\alpha-1}B_0\Big(\frac t{\eps^\alpha}\Big),
$$
where $b_0\in\mathbb R^d$ is a constant vector, and $B_0(s)$ is a continuously differentiable $1$-periodic vector-function. \\
If $1<\alpha<2$, then
$$
b^\eps(t)=\eps^{-1} {b_0}t + \sum\limits_{j=1}^k \eps^{(-1+j(2-\alpha))}b_j t,
$$
where $k=\Big[\frac1{2-\alpha}\Big]$, and $b_j\in\mathbb R^d$, $j=0,1,\ldots,k$.\\
If $\alpha=1$, then
$$
b^\eps(t)=\frac {b_0}\eps t+b_1t+B_0\Big(\frac t{\eps}\Big).
$$
\end{theorem}

\section{ Auxiliary problems and moving frame}\label{Proof-1}

\subsection{Ansatz for solution}\label{sec3_1}

In this section we construct an approximation for a solution  $u^\varepsilon$ of problem \eqref{th-2}. Denote by
$\mathcal{S}(\mathbb{R}^d)$  the Schwartz class of functions in $\mathbb R^d$. In what follows the symbol $:$ stands
for the inner product of matrices and higher order tensors, while $\cdot$ and $\otimes$ denote the inner product of vectors
in $\mathbb R^d$ and the tensor product, respectively.
For a given $u \in C^1((0,T),{\cal{S}}(\mathbb R^d))$  we introduce the following ansatz:
\begin{equation}\label{w_eps}
\begin{array}{l}
\displaystyle
w^{\varepsilon}(x,t)  =  u ( x^\varepsilon\!,  t )
+ \big[ \varepsilon \chi_1 (\frac{x}{\varepsilon}, \frac{t}{\varepsilon^\alpha}) +   \varepsilon^{\gamma_2} \chi_2 (\frac{x}{\varepsilon}, \frac{t}{\varepsilon^\alpha})+ \ldots
\\[3mm] \displaystyle
+ \varepsilon^{\gamma_k} \chi_k (\frac{x}{\varepsilon}, \frac{t}{\varepsilon^\alpha}) + \varepsilon^{\gamma_{k+1}} \chi_{k+1} (\frac{x}{\varepsilon}, \frac{t}{\varepsilon^\alpha}) \big]\!\cdot\! \nabla u  ( x^\varepsilon\!,  t )
+ \varepsilon^2 \varkappa (\frac{x}{\varepsilon}, \frac{t}{\varepsilon^\alpha}):\!\nabla \nabla u  ( x^\varepsilon\!,t ),
\end{array}
\end{equation}
where
$$
\gamma_j = 1 +(j-1)(2-\alpha), \; j = 1, \ldots, k+1, \quad k = k(\alpha) = \Big[ \frac{1}{2-\alpha} \Big].
$$
Another ingredient that we need to construct an approximate solution is  moving coordinates defined by
\begin{equation}\label{G}
x^\varepsilon = x^\varepsilon(t) = x -  b^\varepsilon(t),  \quad
b^\varepsilon(t)= b_0^\eps(t)+ b_1^\eps(t) + \ldots + b^\eps_{k}(t), \quad b^\eps(t) \in \mathbb{R}^d.
\end{equation}
The fact that the function $b^\eps(t)$ is defined as the sum in (\ref{G}) is a consequence of representation (\ref{w_eps}) for the ansatz $w^\eps$.

If  $0 < \alpha <1$, then $k=k(\alpha)=0$, and the formula for $b^\varepsilon(t)$ in \eqref{G} reads
$b^\eps (t) = b^\eps_0(t)$. We will prove that
\begin{equation}\label{bw-1}
b^\eps (t) = b^\eps_0(t)= \frac{b_0}{\eps}t + \eps^{\alpha-1}B_0 \Big(\frac t{\eps^\alpha}\Big),
\end{equation}
where $B_0(s)$ is a bounded periodic Lipschitz continuous function. In this case
 the ansatz $w^\eps$ takes the following form:
\begin{equation}\label{w_eps1}
w^{\varepsilon}(x,t)  =  u ( x^\varepsilon\!,  t )
+  \varepsilon \chi_1 (\frac{x}{\varepsilon}, \frac{t}{\varepsilon^\alpha}) \!\cdot\! \nabla u  ( x^\varepsilon\!,  t ) + \varepsilon^2 \varkappa (\frac{x}{\varepsilon}, \frac{t}{\varepsilon^\alpha}):\!\nabla \nabla u( x^\varepsilon\!,t ).
\end{equation}

If $1 \le \alpha <2$, then $k>0$ and, in addition to $ b^\eps_0(t)$, the
sum $ b^\eps_0(t)+ b^\eps_1(t)+\ldots+ b^\eps_k(t)$ has at least one more term. We will show that
\begin{equation}\label{bw-2}
b_{j}^\eps(t) = \eps^{\gamma_{j}- \alpha}b_j t + o(1), \ j=1, \ldots, k;
\end{equation}
here $b_j \in \mathbb{R}^d, \ j=1, \ldots, k,$ are constant vectors that are defined below.

The correctors
$ \chi_j(\xi, s) = \{ \chi^i_j(\xi,s), i = 1, \ldots, d \}, \ j=1, \ldots, k+1,$ are vector functions from $L^\infty\big(\mathbb T^1; (L^2(\mathbb T^d))^d\big)$, and $\varkappa(\xi, s) = \{ \varkappa^{ij}(\xi,s), i,j = 1, \ldots, d\} \in
L^\infty\big(\mathbb T^1; (L^2(\mathbb T^d))^{d^2}\big)$.

Our goal is to construct the correctors $\chi_1, \ldots, \chi_{k+1}$, $\varkappa$ and the moving frame 
$(x,t)\to (x-b^\eps(t),t)$ in such a way
 that
\begin{equation}\label{ansatz_rough}
H^{\varepsilon} w^{\varepsilon}(x,t) = \partial_t w^{\varepsilon}(x,t)-\!(L^\eps(t) w^{\varepsilon})(x,t)\approx \partial_t u(x^\eps,t)-\mathrm{div} \big(\Theta \nabla u\big)(x^\eps,t) .
\end{equation}

In what follows for brevity we use the notation
$\partial_s v(x,\frac t{\eps^\alpha})=\partial_s v(x,s)\big|_{s=\frac t{\eps^\alpha}}$.
Substituting $w_\eps$ for $u$ in \eqref{ANA_eps} and using (\ref{G}) we have
\begin{equation}\label{Aw}
H^{\varepsilon} w^{\varepsilon}(x,t) \
= \partial_t w^{\varepsilon}(x, t) - L^{\varepsilon}(t) w^{\varepsilon}(x,t)
\end{equation}
with
\begin{equation}\label{Aw-1}
\begin{array}{l}
\displaystyle
\partial_t w^{\varepsilon}(x, t)\! =
\partial_t u (x^\varepsilon\!\!, t)+\! \Big[\!- \frac{d b^\eps(t)}{dt}\!
+\! \sum_{j=1}^{k+1} \varepsilon^{\gamma_j-\alpha} \partial_s \chi_j (\frac{x}{\varepsilon},\frac t{\eps^\alpha})
\Big]\! \cdot\! \nabla u (x^\varepsilon\!\!, t)
 \\[4mm] \displaystyle
 +
 \Big[\sum_{j=1}^{k+1} \eps^{\gamma_j} \chi_j (\frac{x}{\varepsilon}, \frac{t}{\varepsilon^\alpha})
\Big]\!\cdot\! \Big(
- \frac{d \, b^\varepsilon(t)}{dt}
\nabla \nabla u (x^\varepsilon, t)
 + \partial_t \nabla u  ( x^\varepsilon\!,  t ) \Big)
 \\[5mm] \displaystyle
 +  \varepsilon^{2-\alpha} \partial_s \varkappa (\frac{x}{\varepsilon},\frac{t}{\varepsilon^\alpha})    :
\nabla \nabla u (x^\varepsilon, t)
\\[4mm] \displaystyle
+ \varepsilon^2 \varkappa \big(\frac{x}{\varepsilon}, \frac{t}{\varepsilon^\alpha}\big) \! : \!
\Big(\! - \!\frac{d \, b^\varepsilon(t)}{dt} \, \nabla \nabla \nabla u (x^\varepsilon\!, t)\! + \! \partial_t \nabla \nabla  u (x^\varepsilon\!, t)\Big)
\\[4mm] \displaystyle
= \partial_t u (x^\varepsilon, t) + \Big[ - \frac{d  b^\eps(t)}{dt}   +
+\sum_{j=1}^{k}\varepsilon^{\gamma_j-\alpha} \partial_s \chi_j (\frac{x}{\varepsilon}, \frac{t}{\varepsilon^\alpha})
 \Big] \cdot \nabla u (x^\varepsilon, t)
 \\[4mm] \displaystyle
 - \frac{d \, b^\varepsilon(t)}{dt} \otimes
 \varepsilon \chi_1 (\frac{x}{\varepsilon}, \frac{t}{\varepsilon^\alpha})
\! : \!
\nabla \nabla u (x^\varepsilon, t) + \phi_1^\eps(x,t).
 \end{array}
 \end{equation}
 Here
 \begin{equation}\label{restphi-1}
\begin{array}{l}
\displaystyle
  \phi_1^\eps(x,t) =
  \varepsilon^{\gamma_{k+1}-\alpha} \partial_s \chi_{k+1} (\frac{x}{\varepsilon}, \frac{t}{\varepsilon^\alpha}) \cdot \nabla u (x^\varepsilon, t)
  \\[4mm] \displaystyle
 - \frac{d  b^\varepsilon(t)}{dt} \otimes
 \Big[ \sum_{j=2}^{k+1}\varepsilon^{\gamma_j} \chi_j (\frac{x}{\varepsilon},\frac{t}{\varepsilon^\alpha})
\Big]\! : \!
\nabla \nabla u (x^\varepsilon, t)
 \\[4mm] \displaystyle
+ \Big[ \sum_{j=1}^{k+1}\varepsilon^{\gamma_j} \chi_j (\frac{x}{\varepsilon}, \frac{t}{\varepsilon^\alpha})
\Big] \cdot
 \partial_t \nabla u  ( x^\varepsilon\!,  t )
  + \varepsilon^{2-\alpha} \partial_s \varkappa (\frac{x}{\varepsilon},\frac{t}{\varepsilon^\alpha})  \! : \!
\nabla \nabla u (x^\varepsilon, t)
\\[4mm] \displaystyle
+ \varepsilon^2 \varkappa \big(\frac{x}{\varepsilon}, \frac{t}{\varepsilon^\alpha}\big) \! : \!
\Big(\! - \!\frac{d \, b^\varepsilon(t)}{dt} \, \nabla \nabla \nabla u (x^\varepsilon\!, t)\! + \! \partial_t \nabla \nabla  u (x^\varepsilon\!, t)\Big).
 \end{array}
 \end{equation}

After change of variables $z = \frac{x-y}{\varepsilon} = \frac{x^\varepsilon - y^\varepsilon}{\varepsilon}$ and $\xi = \frac{x}{\varepsilon}$,  $\xi \in \mathbb{T}^d$, we obtain
$$
(L^{\eps}(t) w^{\varepsilon})(x,t) =  \frac{1}{\varepsilon^{2}} \int\limits_{\mathbb R^d} a(z) \mu \big(\frac{x}{\eps}, \frac{x}{\eps}-z, \frac{t}{\eps^\alpha}\big) (w^{\varepsilon}(x-\eps z,t) - w^{\varepsilon}(x,t)) dz=
$$
\begin{equation}\label{Aw-2}
\begin{array}{l}
\displaystyle
 =
\frac{1}{\varepsilon^{2}} \int\limits_{\mathbb R^d} a(z) \mu \big( \frac{x}{\varepsilon}, \frac{x}{\varepsilon}-z, \frac{t}{\eps^\alpha} \big)
\bigg\{ (u(x^\varepsilon-\eps z,t) - u(x^\varepsilon,t))
\\[3mm] \displaystyle
+  \Big[\sum_{j=1}^{k+1}\eps^{\gamma_j} \chi_j(\frac{x}{\varepsilon}-z, \frac{t}{\eps^\alpha})
\Big]\!
\cdot\! \nabla u(x^\varepsilon-\eps z,t)
-  \Big[\sum_{j=1}^{k+1}\eps^{\gamma_j} \chi_j(\frac{x}{\varepsilon}, \frac{t}{\eps^\alpha})
 \Big]\!
\cdot\! \nabla u(x^\varepsilon,t)
\\[3mm] \displaystyle +\
\varepsilon^2 \varkappa (\frac{x}{\varepsilon}-z, \frac{t}{\eps^\alpha}) : \nabla \nabla u(x^\varepsilon-\eps z,t) -
 \varepsilon^2 \varkappa (\frac{x}{\varepsilon}, \frac{t}{\eps^\alpha}) : \nabla \nabla u(x^\varepsilon,t) \bigg\} \ dz
 \\[3mm] \displaystyle
 =\bigg\{ \frac{1}{\varepsilon} \int\limits_{\mathbb R^d} a(z) \mu \big(\xi, \xi-z, \frac{t}{\eps^\alpha} \big)
\big\{ -z + \chi_1(\xi-z, \frac{t}{\eps^\alpha})- \chi_1(\xi, \frac{t}{\eps^\alpha}) \big\} dz
\\[3mm] \displaystyle
+ \sum_{j=2}^{k+1} \varepsilon^{\gamma_j - 2}\!\! \int\limits_{\mathbb R^d} a(z) \mu \big(\xi, \xi\!-\!z, \frac{t}{\eps^\alpha} \big)
\big(\chi_j(\xi-z, \frac{t}{\eps^\alpha})- \chi_j(\xi, \frac{t}{\eps^\alpha})\big) dz
\bigg\}\!
\cdot\! \nabla u(x^\eps,t)
\\[3mm] \displaystyle +\ \int\limits_{\mathbb R^d} a(z) \mu \big(\xi, \xi-z, \frac{t}{\eps^\alpha} \big) \Big\{ \frac12 z \otimes z - \chi_1(\xi -z,  \frac{t}{\eps^\alpha} ) \otimes z
\\[3mm] \displaystyle
+ \big( \varkappa (\xi-z, \frac{t}{\eps^\alpha})
 - \varkappa (\xi, \frac{t}{\eps^\alpha})\big) \Big\} \ dz : \nabla \nabla u(x^\eps,t)\ + \ \phi^\varepsilon_2(x,t)
\end{array}
\end{equation}
with the remainder
\begin{equation}\label{14}
\begin{array}{l}
\displaystyle
\phi^\varepsilon_2 (x,t)  =   \int\limits_{\mathbb R^d}
\bigg\{ \int\limits_0^{1} \big( \nabla \nabla u(x^\varepsilon-\varepsilon zq, t) - \nabla \nabla u(x^\varepsilon,t) \big) \! : \! z\!\otimes\!z \,(1-q) \ dq
\\[4mm]   \displaystyle
+\, \varepsilon \chi_1 \big(\xi\!-\!z, \frac{t}{\eps^\alpha} \big)\!\cdot\! \int\limits_0^{1}\!  \nabla \nabla \nabla u(x^\varepsilon\!-\!\varepsilon zq, t) z\!\otimes\!z (1\!-\!q) \, dq
\\[3mm]  \displaystyle +
\Big[ \sum_{j=2}^{k+1}\varepsilon^{\gamma_j -1} \chi_j (\xi\!-\!z, \frac{t}{\varepsilon^\alpha})
\Big] \cdot
\int\limits_0^{1}\!  \nabla \nabla  u(x^\varepsilon\!-\!\varepsilon zq, t)\, z \, dq
\\[4mm]   \displaystyle
- \, \varepsilon \varkappa \big(\xi\!-\!z, \frac{t}{\eps^\alpha} \big) \! : \! \int\limits_0^{1}\!  \nabla \nabla \nabla u(x^\varepsilon\!-\!\varepsilon zq, t) \, z  \, dq\!  \bigg\} \, a (z) \mu \big( \xi, \xi\! -\!z; \frac{t}{\eps^\alpha} \big) \, dz.
\end{array}
\end{equation}

In \eqref{Aw-1} and \eqref{Aw-2} we collected the terms which formally do not  vanish, as $\eps\to0$;  all other terms
form the functions  $\phi^\varepsilon_1$ and  $\phi^\varepsilon_2$.
In the following sections we construct 
 the correctors $\chi_j(\xi,s), \, j=1, \ldots, k+1$, and $\varkappa(\xi,s)$ as well as the vector-function $b^\eps(t)$.
 Then we prove in Lemma \ref{phi} that the $L^\infty((0,T);L^2(\mathbb R^d))$ norm of the remainders $\phi^\varepsilon_1(x,t)$ and $\phi^\varepsilon_2(x,t)$  vanishes, as $\eps\to0$, for  any $u \in C^{1}\big( (0,T), {\cal{S}}(\mathbb R^d) \big)$.
%

\medskip
We begin by collecting in \eqref{Aw-1} and \eqref{Aw-2} power-like terms with non-positive power of $\eps$ and,
for each such power, equating the sum of these terms to 0.

\subsection{Terms of order $\varepsilon^{-1}$} \label{SS-1}

Collecting the terms of order $\varepsilon^{-1}$ in \eqref{Aw-1}, \eqref{Aw-2},
we derive from \eqref{Aw} the following problem for the first corrector $\chi_1 (\xi, s), \xi  \in \mathbb{T}^d, s
\in \mathbb{R}_+$:
\begin{equation}\label{Fcorr}
\int\limits_{\mathbb R^d}  a (z) \mu( \xi, \xi -z; \frac{t}{\eps^\alpha}) \Big( -z + \chi_1 (\xi-z, \frac{t}{\eps^\alpha}) - \chi_1 (\xi, \frac{t}{\eps^\alpha}) \Big) \, dz + \eps \, \frac{d b^\eps_0(t)}{dt} = 0
 \end{equation}
written in the vector form. Since \eqref{Fcorr} is a system of uncoupled equations,
we can consider separately the equation for each component $\chi^i_1(\xi,s), \, i= 1, \ldots, d$, that reads
\begin{equation}\label{Fcorr-1}
 A(s) \chi^i_1 (\xi, s) = f^i( \xi, s)- F_1^i(s), \quad  s=\frac{t}{\eps^\alpha}, \quad\chi_1^i(\cdot,s)\in L^2(\mathbb T^d)
\end{equation}
with
\begin{equation}\label{Acorr-1}
A(s) v (\xi, s) = \int\limits_{\mathbb R^d}  a (z) \mu( \xi, \xi -z; s)( v (\xi-z, s) - v (\xi, s) )  \, dz,
 \end{equation}
\begin{equation}\label{fcorr-1}
f^i (\xi, s) =   \int\limits_{\mathbb R^d}  z^i \, a (z) \, \mu( \xi, \xi -z; s) \, dz, 
 \end{equation}
 here $s$ is a parameter. Since $\chi^j_1(\cdot,s)$, $j=1,\ldots,d$,  is uniquely defined up to an additive constant, we impose
 the condition
 $$
 \int_{\mathbb T^d}  \chi^j_1(\xi,s)d\xi =0
 $$
 in order to make $\chi^j_1(\cdot,s)$ uniquely defined.
  According to Proposition 4.3 in \cite{AA}, the operator $A(s)$ is Fredholm for each $s\in \mathbb R$ and the kernels
  of $A(s)$ and $A^{*}(s)$  have dimension $1$.
  Therefore,
the solvability condition for  \eqref{Fcorr}--\eqref{fcorr-1}
reads
\begin{equation}\label{compat_condi}
 \int\limits_{\mathbb T^d}\big(f^i( \xi, s)- F_1^i(s)\big)p(\xi,s)d\xi=0, \ \ \hbox{for any }s\in\mathbb R,
\end{equation}
where $p(\xi,s) \in  L^2(\mathbb{T}^d)$ is a unique
 solution of the equation
\begin{equation}\label{aux_p_adj}
A^\star (s) p(\cdot, s) = 0, \qquad \int_{\mathbb{T}^d} p(\xi, s)\, d\xi =1,
\end{equation}
with
$$
A^\star (s) p(\xi, s)=\int_{\mathbb R^d}a(\eta-\xi)\mu(\eta,\xi,s) p(\eta,s)d\eta-
 \Big(\int_{\mathbb R^d}a(\xi-\eta)\mu(\xi,\eta,s)d\eta\Big) p(\xi,s).
$$
As was shown in \cite[Corollary 4.1]{AA}, $p(\xi,s)$ satisfies the estimate
\begin{equation}\label{estimo_p}
0<\pi_1\leqslant p(\cdot,s)\leqslant \pi_2,
\end{equation}
where $\pi_1$ and $\pi_2$ do not depend on $s$.
Condition \eqref{compat_condi} implies the following choice of the vector function $F_1(s)$:
\begin{equation}\label{SCb}
F_1(s) = \int\limits_{\mathbb R^d} \int\limits_{\mathbb T^d} z \, a (z) \, \mu( \xi, \xi -z; s) \, p(\xi,s) \, d\xi dz.
\end{equation}
On the other hand, equation \eqref{Fcorr} yields
\begin{equation*}\label{b1_bbbis}
\frac{d b^\eps_0(t)}{dt} = \frac{1}{\eps} \, F_1(\frac{t}{\eps^\alpha})
\end{equation*}
with $F_1(s)$ defined by (\ref{SCb}). Let us notice that the function $F_1(s)$ is periodic with period 1. Denote
\begin{equation}\label{b1-0}
b_0 = \int_0^1 F_1(s) ds, \quad F_1(s) = b_0 + \beta_0(s), \; \mbox{where } \; \int_0^1 \beta_0(s) ds = 0.
\end{equation}
Then
\begin{equation}\label{b1}
\frac{d b^\eps_0(t)}{dt} = \frac{1}{\eps} \, \big( b_0 + \beta_0( \frac{t}{\eps^\alpha}) \big),
\end{equation}
and a solution of \eqref{b1} takes the form
\begin{equation}\label{b1-bis}
b_0^\eps(t) = \frac{b_0}{\eps} \, t + \eps^{\alpha-1} \int_0^{t/\eps^\alpha} \beta_0(\tau) d\tau = \frac{b_0}{\eps} \, t + \eps^{\alpha-1} B_0\Big(\frac t{\eps^\alpha}\Big),
\end{equation}
where, due to \eqref{b1-0},
$$
B_0(s) =  \int_0^{s} \beta_0(\tau) d\tau =  \int_{[s]}^{s} \beta_0(\tau) d\tau
$$
is a bounded $1$-periodic Lipschitz continuous function.
Notice that, if $1 < \alpha < 2$, then the second term in \eqref{b1-bis} is vanishing as $\eps \to 0$.  Consequently,
in this case $b_0^\eps(t) =  \frac{b_0}{\eps} \, t + o(1)$.

Since the first corrector $\chi_1$ is defined as a solution of  \eqref{Fcorr} with
$$
b_0^\eps(t) =  \frac{b_0}{\eps} \, t + \eps^{\alpha-1} B_0\big(\frac t{\eps^\alpha}\big),
$$
the sum of the terms of order $\varepsilon^{-1}$  in \eqref{Aw-1} - \eqref{Aw-2} is equal to 0.

Since  the right-hand side of  relation \eqref{Aw-1} depends on the time derivative of the corrector $\chi_1$,
 we should check that $\chi_1$ is a differentiable in time function. In fact, $\chi_1$ inherits the regularity of
 $\mu$ with respect to time.
 \begin{lemma}\label{l_t_deriv}
 Under assumptions \eqref{M1}--\eqref{k_reg_mu} we have $p(\xi,s)\in C^{k+1}(\mathbb T^1;L^2(\mathbb T^d))$,
 $\chi_1(\xi,s)\in C^{k+1}(\mathbb T^1;L^2(\mathbb T^d))$ and $F_1(s)\in C^{k+1}(\mathbb T^1)$.
 \end{lemma}
\begin{proof}
According to Lemma 4.1 in \cite{AA}, for each $s\in\mathbb T^1$ the operator
$A^*(s):L^2(\mathbb T^d)\mapsto L^2(\mathbb T^d)$ has an isolated simple eigenvalue at zero. Since under our standing
assumptions the family $A^*(s)$ is continuous in $s$ in the operator norm, there exists $\varkappa>0$ such that
$\mathrm{dist}(0, \sigma(A^*(s)\setminus\{0\}))\geqslant 2\varkappa$ for all $s\in\mathbb T^1$. Denote by $\Gamma_\varkappa$ the contour $\{\zeta\in\mathbb C\,:\, |\zeta|=\varkappa\}$.

As a consequence of  \eqref{M1}--\eqref{k_reg_mu} we have $A^*(s)\in
C^{k+1}(\mathbb T^1\,;\,\mathcal{L}(L^2(\mathbb T^d), L^2(\mathbb T^d)))$.
Then for all $\zeta$ from a $\frac\varkappa4$-neighbourhood of $\Gamma_\varkappa$ the resolvent $(A^*(t))-\zeta)^{-1}$
as a function of $s$ also belongs to $C^{k+1}(\mathbb T^1\,;\,\mathcal{L}(L^2(\mathbb T^d), L^2(\mathbb T^d)))$.
Therefore, by the Riesz formula, the projector
\begin{equation}\label{proje}
\mathcal{P}_p(s)=-\frac1{2\pi i}\int_{\Gamma_\varkappa}(A^*(s))-\zeta)^{-1}\,d\zeta= \big(\cdot, \mathtt{1}\big)p(\cdot,s)
\end{equation}
is a $C^{k+1}$ operator family in $L^2(\mathbb T^d)$; here $p(\cdot,s)$ is a function introduced in \eqref{aux_p_adj},
and $\mathtt{1}$ is the function identically equal to $1$.
We recall that $\int_{\mathbb T^d}p(\xi,s)d\xi=1$. Substituting in \eqref{proje} $\mathtt{1}$ for a test function we conclude that $p(\cdot,s)\in C^{k+1}(\mathbb T^1; L^2(\mathbb T^d))$. This yields the first and, in view of \eqref{SCb}, the third statements of the lemma.

Letting $\check A(s)=p(\cdot,s)A(s)(p(\cdot,s))^{-1}$ and considering estimate \eqref{estimo_p},  one can easily check that $\check A(s)$ maps
$\mathop{L}\limits^\circ{}^{\!2}(\mathbb T^d)$ into itself, where  $\mathop{L}\limits^\circ{}^{\!2}(\mathbb T^d)=
\{u\in L^2(\mathbb T^d)\,:\, \int_{\mathbb T^d}u(\xi)d\xi=0\}$.
Moreover, $\check A(s)$ is invertible on this space, and $\check A(s)\in C^{k+1}(\mathbb T^1;
\mathcal{L}(\mathop{L}\limits^\circ{}^{\!2}(\mathbb T^d),\mathop{L}\limits^\circ{}^{\!2}(\mathbb T^d)))$.
Recalling the definition of the function $f(\cdot,s)$ in \eqref{fcorr-1} we conclude that
$\chi_1(\cdot,s)=(p(\cdot,s))^{-1}(\check A(s))^{-1}p(\cdot,s)( f(\cdot,s)-F_1(s))\in C^{k+1}(\mathbb T^1;L^2(\mathbb T^d))$.
\end{proof}

\subsection{Terms of order $\varepsilon^{-\delta}, \ 0 \le \delta<1$} 

The first corrector $ \chi_1$ has been constructed in such a way that the sum of the terms of order $\eps^{-1}$
in \eqref{Aw-2} vanishes.
However, the expression for the time derivative of $w^\eps$ in \eqref{Aw-1} contains the term
$\varepsilon^{1-\alpha} \partial_s \chi_1 \big(\frac{x}{\varepsilon},\frac{t}{\varepsilon^\alpha}\big)  \nabla u (x^\varepsilon, t)$, which is of order $\varepsilon^{1-\alpha}$.
%
If $0<\alpha<1$, then this term is vanishing as $\eps \to 0$ and we don't need additional correctors, except for $\varkappa$.
In this case  the  ansatz $w^\eps$ and the function $b^\eps(t) = b^\eps_0(t)$ are defined by \eqref{w_eps1} and \eqref{bw-1}, respectively,

If $1 \le \alpha <2$, then $1-\alpha \le 0$ and we have to construct higher order correctors to compensate this term and to make the sum of  terms of order $\eps^{1-\alpha}$ in \eqref{Aw-1} - \eqref{Aw-2} equal to zero. This leads to the following problem for the second corrector $\chi_2 (\xi, s)$: 
\begin{equation}\label{chi2}
\begin{array}{l}
\displaystyle
\eps^{\gamma_2-2} \int\limits_{\mathbb R^d}  a (z) \mu( \xi, \xi -z; \frac{t}{\eps^\alpha}) \Big( \chi_2 (\xi-z, \frac{t}{\eps^\alpha}) - \chi_2 (\xi, \frac{t}{\eps^\alpha}) \Big) dz
\\[3mm] \displaystyle
=   \eps^{1-\alpha} \partial_s \chi_1(\xi, \frac{t}{\eps^\alpha}) - \frac{d b^\eps_1(t)}{dt},
\end{array}
 \end{equation}
 or equivalently,
\begin{equation}\label{chi2-1}
 A\Big(\frac{t}{\eps^\alpha}\Big) \chi_2\Big(\xi, \frac{t}{\eps^\alpha}\Big) = \eps^{3- \alpha - \gamma_2}\partial_s \chi_1\Big(\xi, \frac{t}{\eps^\alpha}\Big) -  \eps^{2-\gamma_2}\frac{d b^\eps_1(t)}{dt}.
\end{equation}
Consequently, $\gamma_2 = 3-\alpha$, and repeating the reasoning from the previous subsection we conclude that the solvability condition of equation  \eqref{chi2-1} 
reads
\begin{equation}\label{SCb2}
\eps^{\alpha - 1}\frac{d b^\eps_1(t)}{dt} = \int\limits_{\mathbb T^d} \partial_s \chi_1\big(\xi, \frac{t}{\eps^\alpha}\big) \,
p\big(\xi, \frac{t}{\eps^\alpha}\big) \, d\xi,
\end{equation}
where $p(\xi,s)$ is defined in \eqref{aux_p_adj}. 
Denote by $F_2(s)$ the function on the right-hand side of equation \eqref{SCb2}:
$$
F_2(s) = \int\limits_{\mathbb T^d} \partial_s \chi_1(\xi, s) \, p(\xi, s) \, d\xi.
$$
Then $F_2(s)$ is a periodic function, and we define
$$
b_1 = \int_0^1 F_2(s) ds, \quad F_2(s) = b_1 + \beta_1(s), \; \mbox{where } \; \int_0^1 \beta_1(s) ds = 0.
$$
Then equation \eqref{SCb2} takes the form
\begin{equation}\label{b1-t}
\frac{d b^\eps_1(t)}{dt} = \eps^{1- \alpha} \Big( b_1 + \beta_1(\frac{t}{\eps^\alpha}) \Big),
\end{equation}
and we infer that
\begin{equation}\label{b1-bbis}
b_1^\eps(t) = \eps^{1-\alpha} b_1 t + O(\eps).
\end{equation}
With this definition of the second corrector $\chi_2$,
the sum of  terms of order $\varepsilon^{1- \alpha}$  in \eqref{Aw-1}, \eqref{Aw-2} is equal to 0.

Following the line of the proof of Lemma \ref{l_t_deriv} one can show that
$\chi_2(\cdot,s)\in C^k(\mathbb T^1;L^2(\mathbb T^d))$.

At the next step we deal with the terms of order $\eps^{\gamma_2-\alpha}$.  If $\gamma_2-\alpha<0$, then we should
compensate the term $\varepsilon^{\gamma_2-\alpha} \partial_s \chi_2 (\xi, \frac{t}{\varepsilon^\alpha})$ with the help
of the third corrector $\chi_3(\xi, s)$.   We leave the details to the reader. Notice that
$\chi_3(\cdot,s)\in C^{k-1}(\mathbb T^1;L^2(\mathbb T^d))$.

We iterate this procedure until $\gamma_{k+1}-\alpha > 0$. 
This yields
$$
k = \Big[\frac{1}{2-\alpha}  \Big], \quad \mbox{ and } \; \gamma_k = 1+ (k-1)(2-\alpha).
$$
If $\alpha = 2-\frac1k$ for some $k\in\mathbb N$, 
then $\gamma_k=\alpha$, and
$$
\varepsilon^{\gamma_k-\alpha} \partial_s \chi_k \Big(\xi, \frac{t}{\varepsilon^\alpha}\Big) = \partial_s \chi_k \Big(\xi, \frac{t}{\varepsilon^\alpha}\Big). 
$$
In this case $\gamma_{k+1}=2$, and the last term on the right hand side of \eqref{G} takes the form
$$
b_{k}^\eps(t) = b_{k} t + O(\eps^\alpha),
$$
where
$$
b_{k} = \int_0^1 F_{k+1}(s) ds, \quad F_{k+1}(s) = \int\limits_{\mathbb T^d} \partial_s \chi_k(\xi, s) \, p(\xi, s) \, d\xi.
$$
So in this section we constructed the correctors $\chi_j$, $j=1,\ldots, k+1$, and justified
formula  \eqref{b_eps_def}  for $b^\eps(t)$. Our next goal is to define the corrector $\varkappa$.

\section{Proof of Theorem \ref{MT}.}\label{sec_proofttt}

\subsection{Terms of order $\varepsilon^0$}\label{sss41}

It was shown in the previous section that, under a proper choice of periodic correctors
$\chi_1, \chi_2, \ldots, \chi_{k+1} $
in \eqref{w_eps}, the sum of the  terms corresponding to each negative power of $\eps$ 
 as well as the sum of the terms $ \partial_s \chi_k (\xi, s)$ appearing for exceptional values of $\alpha =  2-\frac1k, \; k=1,2,\ldots,$ in representations \eqref{Aw-1}, \eqref{Aw-2} vanishes.

Next we collect all the terms of order $\varepsilon^{0}$  in \eqref{Aw-1}, \eqref{Aw-2}:
\begin{equation}\label{II-1}
\begin{array}{l}
\displaystyle
\partial_t u (x - b^\varepsilon(t), t) \!-\!
\Big( \!\int\limits_{\mathbb R^d}\! a (z) \mu (\xi, \xi\! -\!z; \frac{t}{\varepsilon^\alpha} ) \big( \frac12 z\!\otimes\!z\! - z \!\otimes\!\chi_1 (\xi\!-\!z, \frac{t}{\varepsilon^\alpha})\big)
   dz  \\[3mm]
\displaystyle
+ \eps \chi_1(\xi,\frac{t}{\varepsilon^\alpha} ) \!\otimes\! \frac{d b^\eps (t)}{dt} + A(\frac{t}{\varepsilon^\alpha}) \varkappa (\xi, \frac{t}{\varepsilon^\alpha})
\Big)\! : \! \nabla \nabla u (x - b^\varepsilon(t), t) = 0,
\end{array}
\end{equation}
where the operator $ A(s)$  is defined in \eqref{Acorr-1}. It is worth noting that, according to \eqref{b1},
$$
\begin{array}{l}
\displaystyle
\eps \chi_1\Big(\xi,\frac{t}{\varepsilon^\alpha} \Big) \otimes \frac{d b^\eps (t)}{dt} = \eps \chi_1\Big(\xi,\frac{t}{\varepsilon^\alpha} \Big) \otimes \frac{d b^\eps_0 (t)}{dt}+
O(\eps^{2-\alpha})
\\[3mm]
\displaystyle
=  \chi_1\Big(\xi,\frac{t}{\varepsilon^\alpha} \Big) \otimes \Big(b_0+ \beta_0\big(\frac{t}{\varepsilon^\alpha}\big)\Big) + O(\eps^{2-\alpha}).
\end{array}
$$
Let us define a matrix $\mathrm{\theta}(s) = \big\{ \mathrm{\theta}^{ij} (s) \big\}$ using the solvability condition for the following equation: 
 \begin{equation}\label{II-6bis}
 \begin{array}{l} \displaystyle
 A(s) \varkappa (\xi, s) = \mathrm{\theta}(s)
 - \chi_1\Big(\xi,\frac{t}{\varepsilon^\alpha} \Big) \!\otimes\!  (b_0+ \beta_0(s))
 \\[3mm]
 \displaystyle
 - \int\limits_{\mathbb R^d}\! a (z) \mu (\xi, \xi\! -\!z; s)
 \big( \frac12 z\!\otimes\!z\! - z \!\otimes\!\chi_1 (\xi\!-\!z,s)\big)
  \, dz
\end{array}
\end{equation}
The solvability condition of \eqref{II-6bis} reads
 \begin{equation}\label{II-6}
 \begin{array}{l}
\displaystyle
\mathrm{\theta}(s) =  \int\limits_{\mathbb T^d} \chi_1(\xi,s) p(\xi, s) d \xi  \otimes  (b_0+ \beta_0(s))
\\[3mm]
\displaystyle
+ \int\limits_{\mathbb R^d} \int\limits_{\mathbb T^d} a (z) \mu (\xi, \xi\! -\!z;s) \big( \frac12 z\!\otimes\!z\! - z \!\otimes\!\chi_1 (\xi\!-\!z, s)\big) p(\xi, s)\, dz\, d\xi.
\end{array}
\end{equation}
Due to conditions \eqref{M1}--\eqref{k_reg_mu} and Lemma \ref{l_t_deriv}, the function $\theta(s)$ is  periodic and $k+1$ times continuously differentiable.
With this choice of $\mathrm{\theta}(s)$ equation \eqref{II-6bis} is solvable.
Inserting  its solution $\varkappa$ in 
\eqref{II-1}  we obtain the following "intermediate" effective equation
 \begin{equation}\label{II-9}
 \begin{array}{rl}
 \displaystyle
 \partial_t w^\eps(x,t)- L^\eps w^\eps (x,t)\! =&\!\!\!
\partial_t u (x - b^\varepsilon(t), t) \! -\! \mathrm{\theta}(\frac{t}{\varepsilon^\alpha}) :  \nabla \nabla u (x - b^\varepsilon(t), t)\\[2mm]
+&\!\! \phi^\eps(x,t),\qquad  \phi^\eps(x,t)=\phi_1^\eps(x,t)+\phi_2^\eps(x,t).
\end{array}
\end{equation}

\begin{lemma}\label{l_posidef}
  The matrix $\theta(s)$ is periodic in time variable, bounded  and positive definite, i.e. there exists a constant $\lambda>0$ such that
   $\lambda|\zeta|^2 \leqslant \theta(s)\zeta\cdot\zeta\leqslant \lambda^{-1}|\zeta|^2$
   for all $s\in\mathbb R$  and $\zeta\in\mathbb R^d$.
\end{lemma}

This statement can be proved in the same way as Proposition 5.1 in \cite{AA}. Periodicity of  $\theta(s)$ follows from representation \eqref{II-6}.

Consequently, the matrix
\begin{equation}\label{Theta}
\Theta = \int_0^1 \theta (s) ds
\end{equation}
is also positive definite. 

\subsection{Partial homogenization}

Let $\rho^\eps$ be a solution of the following problem
\begin{equation}\label{apriori-1}
\partial_t  \rho^\eps(x,t) = \theta \Big(\frac{t}{\varepsilon^\alpha} \Big) :  \nabla \nabla \rho^\eps (x, t), \quad \rho^\eps(x,0) = u_0 \in {\cal S}(\mathbb{R}^d).
\end{equation}
Differentiating this equation in spatial variables, considering the fact that $\theta(s)$ does not depend on $x$ and using
the Aronson estimates, see \cite{Arons},
one can easily check that  $\rho^\eps(x,t) \in C^1((0,T), {\cal S}(\mathbb{R}^d))$ for any $T$.
Moreover, the estimates of the corresponding seminorms of $\rho^\eps(x,t)$ and $\partial_t \rho^\eps(x,t)$ are uniform
in $\eps$ and $t\in[0,T]$.
  We then substitute $\rho^\eps$ for $u$ in \eqref{w_eps} and define an ansatz $w^\varepsilon$ by formula \eqref{w_eps}
 with $x^\eps$ given by \eqref{G} and  $\chi_j(\xi,s)$,  $\varkappa(\xi,s)$ being periodic
 solutions of the above auxiliary equations. This formula reads
 \begin{equation}\label{w_rho_eps}
\begin{array}{rl}
\displaystyle
w^{\varepsilon}(x,t)  = &  \displaystyle
\rho^\eps ( x^\varepsilon\!,  t )
+ \Big[ \sum_{j=1}^{k+1}\varepsilon^{\gamma_j} \chi_j \Big(\frac{x}{\varepsilon}, \frac{t}{\varepsilon^\alpha}\Big)
 \Big]\!\cdot\! \nabla \rho^\eps  ( x^\varepsilon\!,  t )
\\[3mm]
\displaystyle
& \displaystyle
+ \varepsilon^2 \varkappa \Big(\frac{x}{\varepsilon}, \frac{t}{\varepsilon^\alpha}\Big):\!\nabla \nabla \rho^\eps  ( x^\varepsilon\!,t ),
\end{array}
\end{equation}

  It follows from \eqref{II-9} that
\begin{equation}\label{w_eps_bis}
\begin{array}{l}
\displaystyle
\partial_t w^{\varepsilon}(x,t) -  L^{\varepsilon} w^{\varepsilon}(x,t)
\\[3mm] \displaystyle
=
\partial_t \rho^\eps(x^\varepsilon,t) -
 \theta \Big(\frac{t}{\varepsilon^2} \Big) : \nabla \nabla \rho^\eps(x^\varepsilon,t) \ + \ \phi^\varepsilon(x,t) =
\phi^\varepsilon(x,t),
\\[3mm] \displaystyle
 w^{\varepsilon}(x,0) = u_0(x) + \psi^\varepsilon(x),
 \end{array}
\end{equation}
where $x^\eps = x - b^\varepsilon(t)$.
The remainder term $\phi^\varepsilon(x,t)=\phi_1^\eps(x,t)+\phi_2^\varepsilon(x,t)$
is defined in \eqref{restphi-1} and \eqref{14}, and the function $\psi^\eps$ is given by
$$
\begin{array}{l}
 \displaystyle
\psi^\varepsilon (x) = \Big(\varepsilon \chi_1 (\frac{x}{\varepsilon},0)+ \eps^{\gamma_2}\chi_2(\frac{x}{\varepsilon},0) + \ldots +  \eps^{\gamma_{k+1}}\chi_{k+1}(\frac{x}{\varepsilon},0) \Big) \cdot \nabla u_0(x)
\\[3mm] \displaystyle
+ \varepsilon^2 \varkappa (\frac{x}{\varepsilon},0) : \nabla \nabla u_0 (x),
\end{array}
$$
$k = \big[ \frac{1}{2-\alpha} \big], \ \gamma_j = 1+ (j-1)(2-\alpha), \ j=1, \ldots,k+1$.

Consequently, the difference  $v^\varepsilon = w^\varepsilon - u^\varepsilon$, where $u^\varepsilon$ is the solution of \eqref{th-2},
satisfies the following problem:
\begin{equation}\label{v_eps_bis}
\partial_t v^{\varepsilon}(x,t) -  L^{\varepsilon} v^{\varepsilon}(x,t) = \phi^\varepsilon(x,t), \quad v^{\varepsilon}(x,0) =  \psi^\varepsilon(x).
\end{equation}
It is straightforward to check that, for any $u_0\in \mathcal{S}(\mathbb R^d)$,
\begin{equation}\label{small_ini}
  \|\psi^\eps\|_{L^2(\mathbb R^d)}\ \to\ 0,\quad\hbox{as }\eps\to0.
\end{equation}
In the following statement we provide estimates for the function $\phi^\eps(x,t)$.
\begin{lemma}\label{phi} Let $u \in C^{1}\big( (0,T), {\cal{S}}(\mathbb R^d) \big)$.
Then for the functions $ \phi^\varepsilon_1$ and $ \phi^\varepsilon_2$ given by \eqref{restphi-1} and \eqref{14} we have
\begin{equation}\label{fi}
\| \phi^\varepsilon_1 \|\big._{L^{\infty}( (0,T); L^2 (\mathbb R^d))}  \to  0 \quad \mbox{and} \quad \| \phi^\varepsilon_2 \|\big._{L^{\infty}( (0,T); L^2 (\mathbb R^d) )}  \to  0, \  \mbox{as} \; \; \varepsilon \to 0.
\end{equation}
\end{lemma}

\begin{proof}
It follows from \eqref{G}, \eqref{b1} and \eqref{b1-t} that $\frac{d b^\eps (t)}{dt}\le C \eps^{-1}$ for all $0<\alpha<2$. Also, according to Lemma \ref{l_t_deriv},  all the components of $\chi_j(\xi, s) , \, \partial_s \chi_j(\xi, s), \; j=1, \ldots, k+1,$ and $\varkappa(\xi, s) , \, \partial_s \varkappa(\xi, s)$ are elements of $L^\infty((0,+\infty);L^2(\mathbb T^d))$.
Since $\rho^\eps$ and $\partial_t \rho^\eps$  are Schwartz class functions of $x$, this yields
$$
\| \phi^\varepsilon_1\|_{L^\infty((0,T);L^2(\mathbb R^d))} \le C_1 \eps^{\delta_1}, \quad \mbox{with} \; \; \delta_1 = \min\{\gamma_{k+1}-\alpha, \, 2-\alpha, \, 1  \}.
$$
As was shown in \cite[Proposition 5]{PZh}, the first, the second and the last terms on the right-hand side of \eqref{14} converge to zero in $L^\infty((0,T);L^2(\mathbb R^d))$,
as $\eps\to0$.
The smallness of the third term  in $L^\infty((0,T);L^2(\mathbb R^d))$  follows from Lemma \ref{l_t_deriv} and the fact that $\rho^\eps(t,\cdot)$ is a Schwartz class function.
\end{proof}

\subsection{A priori estimates}\label{sec_apriori}
Now our goal is to estimate the function $v^{\varepsilon}(\cdot)$.
The estimates for $v^{\varepsilon}(\cdot)$ rely on the following statement.
\begin{proposition}\label{prop_apriori}
A solution of the problem
\begin{equation}\label{au_apriori}
\partial_t \Xi^{\varepsilon}(x,t) -  L^{\varepsilon} \Xi^{\varepsilon}(x,t) = f(x,t), \quad \Xi^{\varepsilon}(x,0) =  g_0(x),
\end{equation}
with $f\in L^2((0,T);L^2(\mathbb R^d))$ and $g_0\in L^2(\mathbb R^d)$ admits the following upper bound
\begin{equation}\label{xixi_est}
\|\Xi^\eps\|_{L^\infty((0,T);L^2(\mathbb R^d))}\leqslant C_1(T)\|f\|_{L^2((0,T);L^2(\mathbb R^d))}
+C_2\|g_0\|_{L^2(\mathbb R^d)}
\end{equation}
with constants $C_1(T)$ and $C_2$ that do not depend on $\eps$.
\end{proposition}
\begin{proof}
In order to show that estimate \eqref{xixi_est} holds we introduce a weighted $L^2(\mathbb R^d)$ space.
Consider an auxiliary Cauchy problem
\begin{equation}\label{aux_qq}
\textstyle
-\partial_s q(\xi,s)= \!A^\star(\eps^{2-\alpha}s)q(\xi,s),\quad q\big(\xi,\frac{2T}{\eps^2}\big)=1,
\quad(\xi,s)\in\mathbb T^d\!\times\!\big(\!-\infty, \frac{2T}{\eps^2}\big).
\end{equation}
We denote a solution of this problem by $q^\eps(\xi,s)$.
Notice that the function 
$q^\eps\big(\frac x\eps, \frac t{\eps^2}\big)$ is $\eps$-periodic in the variables $x_1,\ldots, x_d$ and solves the equation
\begin{equation}\label{eq_check_q}
- \partial_t  v=L^{\eps} (t)^\star v
\end{equation}
with
\begin{equation}\label{eq_check_qqqq}
\begin{array}{c}
\displaystyle
L^{\eps}(t)^\star v(x,t)=
 \frac1{\eps^{d+2}}\int_{\mathbb R^d}a\Big(\frac{y-x}\eps\Big)
\mu\Big(\frac{y}\eps,\frac{x}\eps,\frac{t}{\eps^\alpha}\Big)v(y,t)dy\\[5mm]
\displaystyle
-\frac1{\eps^{d+2}}\bigg(\int_{\mathbb R^d}a\Big(\frac{x-y}\eps\Big)
\mu\Big(\frac{x}\eps,\frac{y}\eps,\frac{t}{\eps^\alpha}\Big)dy\bigg) v(x,t)
\end{array}
\end{equation}
\begin{lemma}\label{l_check_q_bound}
  The function $q^\eps$  satisfies the estimate
  \begin{equation}\label{more_est}
  0<\pi_3\leqslant q^\eps(\xi,s)\leqslant \pi_4<\infty
  \end{equation}
  with the constants $\pi_3$ and $\pi_4$ that do not depend on $\eps$.
\end{lemma}
\begin{proof}
  The proof of this Lemma follows the line of the proof of Lemma 3.4 in \cite{PZ_jmpa25}.
\end{proof}

Multiplying the equation in \eqref{au_apriori} by $\big( q^\eps\big(\frac x\eps,\frac t{\eps^2}) \Xi^\eps(x,t)\big)$ and integrating the resulting relation over
$\mathbb R^d\times(0,t)$, $t\leqslant T$, we obtain
\begin{equation}\label{identity1}
\begin{array}{c}
\displaystyle
\frac12\!\int\limits_0^t\!\!\int\limits_{\mathbb R^d}\!{q}^\eps\big(\frac x\eps,\frac s{\eps^2}\big)\partial_s((\Xi^\eps(x,s))^2) dxds
-\!\!\int\limits_0^t\!\!\int\limits_{\mathbb R^d}\!{q}^\eps\big(\frac x\eps,\frac s{\eps^2}\big)\Xi(x,s)L^{\varepsilon} \Xi^{\varepsilon}(x,s) dxds\\[3mm]
\displaystyle
=\int\limits_0^t\int\limits_{\mathbb R^d}{q}^\eps\big(\frac x\eps,\frac s{\eps^2}\big)f(x,s)\Xi^\eps(x,s)dxds.
\end{array}
\end{equation}
Considering equation \eqref{eq_check_q}
and making straightforward rearrangements, we arrive at the following relation, see also the proof of  \cite[Proposition 6.1]{AA}:
$$
\begin{array}{c}
\displaystyle
\int\limits_{\mathbb R^d}{q}^\eps \big(\frac x\eps,\frac t{\eps^2}\big)(\Xi^\eps(x,t))^2 dx\\[3mm]
\displaystyle
+\frac 1{\eps^{d+2}}\int\limits_0^t\int\limits_{\mathbb R^d}a\Big(\frac{x-y}\eps\Big)
\mu\Big(\frac x\eps,\frac y\eps,\frac s{\eps^\alpha}\Big){q}^\eps \big(\frac x\eps,\frac s{\eps^2}\big)
(\Xi^\eps(x,s)- \Xi^\eps(y,s))^2 dxds\\[3mm]
\displaystyle
=\int\limits_{\mathbb R^d}{q}^\eps(\frac x\eps,0)(g_0(x))^2 dx
+2\int\limits_0^t\int\limits_{\mathbb R^d}{q}^\eps \big(\frac x\eps,\frac s{\eps^2}\big)\Xi^\eps(x,s)dxds.
\end{array}
$$
Sinse $q^\eps$ satisfies estimate \eqref{more_est}, the desired upper bound can be derived from the latter relation by the Gronwall theorem.
\end{proof}

\subsection{Convergence results}\label{ss_conv_fin}

We begin this section by estimating the difference $v^\eps=w^\eps-u^\eps$.
According to  \eqref{small_ini} and Lemma \ref{phi}, both
$\|\phi^\varepsilon \|_{L^{\infty}((0,T);L^2(\mathbb R^d))}$ and $\| \psi^\varepsilon \|_{L^2(\mathbb{R}^d)}$  tend to zero as $\eps\to0$.   Applying Proposition \ref{prop_apriori} to problem
\eqref{v_eps_bis} we conclude that
\begin{equation}\label{apriori-2}
\|v^\eps\|_{L^{\infty}((0,T);L^2(\mathbb R^d))} =
 \| w^\varepsilon - u^\varepsilon\|_{L^{\infty}((0,T);L^2(\mathbb R^d))}  \to 0 \quad \mbox{ as } \;  \eps \to 0.
\end{equation}
On the other hand, considering the structure of the ansatz $w^\eps$ in \eqref{w_rho_eps} one can easily show that
\begin{equation}\label{apriori-3}
 \| w^\varepsilon(x,t) - \rho^\varepsilon(x^\eps, t)\|_{L^{\infty}((0,T);L^2(\mathbb R^d))}  \to 0 \quad \mbox{ as } \;  \eps \to 0.
\end{equation}
Combining the last two limit relations yields
\begin{equation}\label{apriori-3biss}
 \| u^\varepsilon(x,t) - \rho^\varepsilon(x^\eps, t)\|_{L^{\infty}((0,T);L^2(\mathbb R^d))}  \to 0 \quad \mbox{ as } \;  \eps \to 0.
\end{equation}
Let $u^0$ be a solution to problem \eqref{intr_eff} with $\Theta$ defined in \eqref{Theta}.
Our next goal is to show that the solution $\rho^\eps$ of problem \eqref{apriori-1} converges to $u^0$
as $\eps\to0$.
\begin{lemma}\label{l_closedness}
  For any $u_0\in \mathcal{S}(\mathbb R^d)$ the family $\{\rho^\eps(x,t)\}$ converges to $u^0(x,t)$ in $C((0,T);L^2(\mathbb R^d))$ as $\eps\to0$.
\end{lemma}
\begin{proof}
Differentiating equation \eqref{apriori-1} in spatial variables
it is straightforward to obtain the following estimates for $\rho^\eps$:
\begin{equation}\label{ee_compaa}
\|\rho^\eps\|_{L^\infty((0,T);H^1(\mathbb R^d))}+\|\partial_t \rho^\eps\|_{L^\infty((0,T);H^{1}(\mathbb R^d))}
\leqslant C\|u_0\|_{H^3(\mathbb R^d)},
\end{equation}
here the constant $C$ does not depend on $\eps$. Denote $B_R=\{x\in\mathbb R^d\,:\,|x|\leqslant R\}$.
{
Since for any $R>0$ the space $W^{1,\infty}((0,T);H^1(B_R))$ is compactly imbedded to $C((0,T);L^2(B_R)$, estimate \eqref{ee_compaa} implies that for any $R>0$
the family $\{\rho^\eps\}$ is compact in $C((0,T);L^2(B_R))$. According to  \cite{Arons} the fundamental solution of Cauchy problem \eqref{apriori-1}
satisfies the following Aronson upper bound
$$
Q_\eps(x,t)\leqslant C_1t^{-\frac d2}\exp\Big(-c_2 \frac{|x|^2}t\Big).
$$
with the constants $C_1$ and $c_2$ that only depend on the ellipticity constants of the matrix $\theta$.
Therefore, $\|\rho^\eps\|_{C((0,T);L^2(\mathbb R^d\setminus B_R))}$ tends to zero as $R\to\infty$ uniformly in $\eps$.
Combining this estimate with the above compactness result
we conclude that $\{\rho^\eps\}$ is compact in $C((0,T);L^2(\mathbb R^d))$. }
For a converging subsequence of $\{\rho^\eps\}$ denote its limit by $\rho^0$. Let $\psi(x,t)$ be a
$C^\infty((0,T);C_0^\infty(\mathbb R^d))$ function such that $\psi(x,T)=0$. The integral identity of problem \eqref{apriori-1} reads
$$
-\int_0^T\int_{\mathbb R^d}\rho^\eps \partial_t\psi\,dxdt=\int_{\mathbb R^d} u_0 \psi(\cdot,0)\,dx+
\int_0^T\int_{\mathbb R^d}\rho^\eps \theta\Big(\frac t{\eps^\alpha}\Big) : \nabla\nabla\psi\,dxdt
$$
Passing to the limit as $\eps\to0$
in this integral identity 
we obtain the integral identity
of problem  \eqref{intr_eff}. Therefore, $\rho^0=u^0$ and the whole family $\{\rho^\eps\}$ converges to $u^0$:
 \begin{equation}\label{apriori-4}
 \| u^0 - \rho^\varepsilon\|_{L^{\infty}((0,T);L^2(\mathbb R^d))}  \to 0 \quad \mbox{ as } \;  \eps \to 0.
\end{equation}
\end{proof}

Since {
$\|\rho^\eps(x,t)-u^0(x,t)\|_{L^{\infty}((0,T);L^2(\mathbb R^d))}=
\|\rho^\eps(x^\eps,t)-u^0(x^\eps,t)\|_{L^{\infty}((0,T);L^2(\mathbb R^d))}$, then \eqref{apriori-2} - \eqref{apriori-4} yield the convergence
\begin{equation}\label{apriori-5}
 \| u^\eps (x,t) - u^0(x^\varepsilon,t)\|_{L^{\infty}((0,T);L^2(\mathbb R^d))}  \to 0 \quad \mbox{ as } \;  \eps \to 0,
\end{equation}
where $x^\eps$ is defined by formula \eqref{G}.

Thus, we have proved convergence \eqref{th-4-I} for any
$u_0 \in {\cal S}(\mathbb{R}^d)$.

Approximating any $L^2$ initial function by a sequence of ${\cal S}(\mathbb R^d)$ functions and taking into account the a priori estimates obtained in Proposition \ref{prop_apriori}
and similar estimates for the limit problem in \eqref{intr_eff}, we derive the statement of Theorem \ref{MT}.
\hfill$\Box$\\[4mm]


\bigskip\noindent
{\bf Author contributions} \
The authors contributed equally to the manuscript.

\bigskip\noindent
{\bf Funding} \ The work of both authors was partially supported by the project “Pure Mathematics
in Norway” and the UiT Aurora project MASCOT.

\bigskip\noindent
{\bf Data availability}  Data sharing is not applicable to this article as no datasets were generated or analyzed
during the current study.

\subsubsection*{Declarations}

\bigskip\noindent
{\bf Conflict of interest}
The authors declare that they have no conflict of interest regarding the publication of this paper.

\bigskip\noindent
{\bf Ethical approval} Not applicable.

\end{document}